\documentclass[11pt]{article}
\usepackage{enumerate}
\usepackage{amssymb,a4wide,latexsym,makeidx,epsfig}
\usepackage{amsthm}
\usepackage{lipsum} 
\usepackage{amsmath}
\usepackage{enumerate}
\usepackage{mathrsfs}
\usepackage{xcolor}
\usepackage{tikz}
\usepackage{geometry}
\newtheorem{theorem}{Theorem}[section]

\newtheorem{definition}[theorem]{Definition}
\newtheorem{lemma}[theorem]{Lemma}
\newtheorem{fact}[theorem]{Fact}
\newtheorem{proposition}[theorem]{Proposition}
\newtheorem{corollary}[theorem]{Corollary}

\newtheorem{conjecture}[theorem]{Conjecture}
\newtheorem{claim}[theorem]{Claim}

\begin{document}
\textwidth 150mm \textheight 225mm
\title{Extremal problems and results related to Gallai-colorings
\thanks{Supported by the National Natural Science Foundation of China (No. 11871398)
and China Scholarship Council (No. 201906290174).}
}
\author{{Xihe Li$^{1,2}$, Hajo Broersma$^{2,}$\thanks{Corresponding author.}, Ligong Wang$^{1}$}\\
{\small $^{1}$ School of Mathematics and Statistics,}\\ {\small Northwestern Polytechnical University, Xi'an, Shaanxi 710129, PR China}\\
{\small $^{2}$ Faculty of Electrical Engineering, Mathematics and Computer Science,}\\ {\small University of Twente, P.O. Box 217, 7500 AE Enschede, The Netherlands}\\
{\small E-mail: lxhdhr@163.com; h.j.broersma@utwente.nl; lgwangmath@163.com}}
\date{}
\maketitle
\begin{center}
\begin{minipage}{120mm}
\vskip 0.3cm
\begin{center}
{\small {\bf Abstract}}
\end{center}
{\small A Gallai-coloring (Gallai-$k$-coloring) is an edge-coloring (with colors from $\{1, 2, \ldots, k\}$) of a complete graph without rainbow triangles. Given a graph $H$ and a positive integer $k$, the $k$-colored Gallai-Ramsey number $GR_k(H)$ is the minimum integer $n$ such that every Gallai-$k$-coloring of the complete graph $K_n$ contains a monochromatic copy of $H$. In this paper, we consider two extremal problems related to Gallai-$k$-colorings. First, we determine upper and lower bounds for the maximum number of edges that are not contained in any rainbow triangle or monochromatic triangle in a $k$-edge-coloring of $K_n$. Second, for $n\geq GR_k(K_3)$, we determine upper and lower bounds for the minimum number of monochromatic triangles in a Gallai-$k$-coloring of $K_{n}$, yielding the exact value for $k=3$. Furthermore, we determine the Gallai-Ramsey number $GR_k(K_4+e)$ for the graph on five vertices consisting of a $K_4$ with a pendant edge.

 }
\end{minipage}
\end{center}

\section{Introduction}
\label{sec:ch-introduction}

In this paper, we only consider edge-colorings of finite simple graphs. For an integer $k \geq 1$, let $c\colon\, E(G)\rightarrow [k]$ be a $k$-edge-coloring (not necessarily a proper edge-coloring) of a graph $G$, where $[k] := \{1, 2, \ldots, k\}$. A graph with an edge-coloring is called {\it rainbow} if all edges are colored differently, and {\it monochromatic} if all edges are colored the same. A {\it Gallai-$k$-coloring} is a $k$-edge-coloring of a complete graph without rainbow triangles, i.e. at most two distinct colors are assigned to the edges of every copy of $K_3$.

The term {\it Gallai-coloring} was first used by Gy\'{a}rf\'{a}s and Simonyi \cite{GySi} in honor of Gallai's decomposition lemma for rainbow triangle-free colorings \cite{Gallai}, but the study of Gallai-colorings has  arisen in a wide range of areas, such as poset theory \cite{Gallai}, the Erd\H{o}s-Hajnal conjecture \cite{FoGP}, rainbow Erd\H{o}s-Rothschild problem \cite{BaLi,BaBH}, information theory \cite{Ko'Si,KoST}, perfect graph theory \cite{CaEL}, and Ramsey-type problems \cite{GSSS,HMOT}.

Given a positive integer $k$ and graphs $H_1, H_2, \ldots, H_k$, the classical $k$-colored {\it Ramsey number} $R(H_1, H_2, \ldots, H_k)$ is the minimum integer $n$ such that every $k$-edge-coloring of $K_n$ contains a monochromatic copy of $H_i$ in color $i$ for some $i\in[k]$. It is well-known that determining the exact value of the Ramsey number is an extremely difficult problem, even for relatively small graphs. Many variants of Ramsey numbers concerning rainbow structures have been studied, such as rainbow-Ramsey numbers, anti-Ramsey numbers and Gallai-Ramsey numbers. We refer to two surveys \cite{FuMO1,Rad} for more information on these topics.

Given $k$ graphs $H_1, H_2, \ldots, H_k$, the $k$-colored {\it Gallai-Ramsey number} $GR(H_1, H_2, \ldots, H_k)$ is defined to be the minimum integer $n$ such that every Gallai-$k$-coloring of the complete graph on $n$ vertices contains a monochromatic copy of $H_i$ in color $i$ for some $i\in[k]$. In the special case when $H_1=H_2=\cdots=H_k=H$, we simply write $R_k(H)$ and $GR_k(H)$ for $R(H, H, \ldots, H)$ and $GR(H, H, \ldots, H)$, respectively. Gallai-Ramsey theory has been increasingly popular over the past decade. We refer to papers \cite{BrSo,FuMa2,GySi,HMOT,LiWL,LMSSS,ZhSC2} for more information on some related problems.

A natural problem related to Gallai-Ramsey theory is to determine the maximum number of edges that are not contained in any rainbow copy of $K_3$ or monochromatic copy of $H$. The analogous problem for Ramsey numbers was considered in \cite{KeSu,LiPS,Ma}; in these papers the authors studied the maximum number of edges not contained in any monochromatic copy of $H$ over all $k$-edge-colorings of $K_n$. For $k\geq 2$, let $f_{k}(n, H)$ denote the maximum number of edges not contained in any rainbow triangle or monochromatic copy of $H$, over all $k$-edge-colorings of $K_n$. The first part of this paper is devoted to this problem.

Let $ex(n, H)$ be the maximum number of edges of an $H$-free graph of order $n$, i.e., the Tur\'{a}n number of $H$. By Tur\'{a}n's theorem, the unique $K_{r+1}$-free graph on $n$ vertices with $ex(n, K_{r+1})$ edges is the Tur\'{a}n graph $T_r(n)$, i.e., the complete $r$-partite graph
on $n$ vertices with class sizes as equal as possible. Let $t(n,r)$ be the number of edges of $T_r(n)$. Note that we have the trivial upper bound $f_{k}(n, H)\leq t(n, GR_{k}(H)-1)$. We also have a trivial lower bound $f_k(n, H)\geq f_2(n, H)\geq ex(n, H)$. For the case $H=K_3$, we will prove the following theorem.

\noindent\begin{theorem}\label{th:RMk3k3} For any real number $\delta>0$, there exists an $n_0$ such that for all $n\geq n_0$, we have $t(n, GR_{k-1}(K_3)$ $- 1)\leq f_{k}(n, K_3)<t(n, GR_{k-1}(K_3)-1)+\delta n^2$.
\end{theorem}

We conjecture that the lower bound on $f_{k}(n, K_3)$ in Theorem~\ref{th:RMk3k3} is in fact the exact value of $f_{k}(n, K_3)$. Moreover, we can generalize this result to a general graph $H$ (see Theorem \ref{th:RMk3H}).

The second part of this paper is devoted to the Gallai-Ramsey multiplicity problem. By the definition of the Gallai-Ramsey number, if $n\geq GR_k(H)$, then any Gallai-$k$-coloring of $K_{n}$ contains a monochromatic copy of $H$. In fact, there could be more than one monochromatic copy of $H$. In light of this, it is natural to consider the minimum number of monochromatic copies of $H$ (as an unlabeled graph) in a Gallai-$k$-coloring of $K_{n}$. Let $g_{k}(H,n)$ denote the minimum number of monochromatic copies of $H$ taken over all Gallai-$k$-colorings of $K_{n}$. The analogous problem for Ramsey numbers is known as the Ramsey multiplicity problem, that is, to consider the minimum number  $M_k(H,n)$ of monochromatic copies of $H$ taken over all $k$-edge-colorings of $K_{n}$ (see \cite{Con,CKPSTY,Fox,HHKNR} for some recent results). With the additional restriction imposed on Gallai-colorings, it is obvious that $g_{k}(H,n)\geq M_k(H,n)$. In 1959, Goodman \cite{Goo} proved the following classical result concerning $M_2(K_3,n)$.

\noindent\begin{theorem}\label{th:M2} {\normalfont (\cite{Goo})} For any positive integer $n$, we have
\[M_2(K_3,n)= \begin{cases}
                n(n-2)(n-4)/24, & \mbox{if $n$ is even},\\
                n(n-1)(n-5)/24, & \mbox{if $n\equiv 1\bmod{4}$},\\
                (n+1)(n-3)(n-4)/24, & \mbox{if $n\equiv 3\bmod{4}$}.
              \end{cases}\]
\end{theorem}

For the case of $3$-edge-colorings, Cummings, Kr\'{a}l', Pfender, Sperfeld, Treglown and Young \cite{CKPSTY} proved the following result, using flag algebras and a probabilistic argument.

\noindent\begin{theorem}\label{th:M3} {\normalfont (\cite{CKPSTY})}
There exists an integer $n_0$ such that for $n\ge n_0$, if we write $n=5m+r$ for nonnegative integers $m$ and $r$ with $0\leq r\leq 4$, then $$M_3(K_3, n)=r\binom{m+1}{3}+(5-r)\binom{m}{3}.$$
\end{theorem}

Our next result shows that $g_{3}(K_3,n)=M_3(K_3, n)$ if $n$ sufficiently large, and gives upper and lower bounds for $g_{k}(K_3,n)$ for other values of $k$.

\noindent\begin{theorem}\label{th:multiplicity1} For $n\geq GR_k(K_3)$, we write $n=5^{\lfloor(k-1)/2\rfloor}m+r$, where $m$ and $r$ are nonnegative integers with $0\leq r\leq 5^{\lfloor(k-1)/2\rfloor}-1$. Then
$$g_{k}(K_3,n)\leq \left\{
                \begin{aligned}
                & r\binom{m+1}{3}+\left(5^{(k-1)/2}-r\right)\binom{m}{3}, & & \mbox{if $k$ is odd},\\
                & rM_2(K_3, m+1)+\left(5^{(k-2)/2}-r\right)M_2(K_3, m), & & \mbox{if $k$ is even}.
              \end{aligned}
           \right.$$
Moreover, let $s_0=1$ if $k$ is odd, and $s_0=2$ if $k$ is even. Then
$$g_{k}(K_3,n)\geq \frac{s_0n(n-1)(n-2)}{GR_k(K_3)(GR_k(K_3)-1)(GR_k(K_3)-2)}.$$
\end{theorem}

In general, we conjecture that the above upper bound on $g_{k}(K_3,n)$ in Theorem~\ref{th:multiplicity1} is in fact the exact value of $g_{k}(K_3,n)$, but we can only verify this for the following cases: (1) $k=3$ and $n$ sufficiently large, (2) $k\geq 3$ and $n=GR_k(K_3)$, (3) $k$ is odd and $GR_k(K_3)\leq n\leq GR_k(K_3)+5^{(k-1)/2}-1$.

Finally, we consider the original problem, the Gallai-Ramsey number for a graph $H$. In \cite{GSSS}, Gy\'{a}rf\'{a}s, S\'{a}rk\"{o}zy, Seb\H{o} and Selkow provided the following general statement on the value of the Gallai-Ramsey number $GR_k(H)$.

\noindent\begin{theorem}\label{th:general_bounds} {\normalfont (\cite{GSSS})} For any graph $H$ and positive integer $k$, if $H$ is not bipartite, then $GR_k(H)$ is exponential in $k$, and if $H$ is bipartite but not a star, then $GR_k(H)$ is linear in $k$.
\end{theorem}

In \cite{FoGP}, Fox, Grinshpun and Pach posed the following conjecture on an expression for the Gallai-Ramsey numbers of complete graphs in terms of their 2-colored Ramsey numbers.

\noindent\begin{conjecture}\label{conjKt} {\normalfont (\cite{FoGP})}
For integers $k\geq 1$ and $t\geq 3$,
$$GR_k(K_{t})=
\left\{
   \begin{aligned}
    &(R_2(K_t)-1)^{k/2}+1, & & \mbox{if $k$ is even},\\
    &(t-1)\cdot(R_2(K_t)-1)^{(k-1)/2}+1, & & \mbox{if $k$ is odd}.
   \end{aligned}
   \right.$$
\end{conjecture}

The cases with $t=3$ and $t=4$ of the above conjecture were verified in \cite{ChGr,GSSS} and \cite{LMSSS}, respectively. Let $K_4+e$ denote the graph on five vertices consisting of a $K_4$ with a pendant edge. We prove the following related result, confirming that the expression in the above conjecture in fact also holds for $K_4+e$ (taking $t=5$), since $R_2(K_4+e)=18$ by a result in \cite{HaMe}.

\noindent\begin{theorem}\label{th:GRk4+e} For integers $k\geq 1$,
$$GR_k(K_{4}+e)=
\left\{
   \begin{aligned}
    &17^{k/2}+1, & & \mbox{if $k$ is even},\\
    &4\cdot 17^{(k-1)/2}+1, & & \mbox{if $k$ is odd}.
   \end{aligned}
   \right.$$
\end{theorem}

The remainder of this paper is organized as follows. In Section \ref{sec:ch-preliminaries}, we will introduce some additional terminology and notation, and list some known results that will be used in our proofs of the main results. In Section \ref{sec:ch-proof-RMk3k3}, we will prove Theorem \ref{th:RMk3k3}, using a variant of the Gallai-Ramsey number. In Section \ref{sec:ch-proof-multiplicity}, we will consider the Ramsey multiplicity problem for Gallai-colorings and prove Theorem \ref{th:multiplicity1}. In Section \ref{sec:ch-proof-GRk4+4}, we will prove Theorem \ref{th:GRk4+e} in a more general form. Finally, we will conclude the paper with some remarks and open problems in Section \ref{sec:ch-remark}.

\section{Preliminaries}
\label{sec:ch-preliminaries}

We begin with the following structural result on Gallai-colorings of complete graphs.

\noindent\begin{theorem}\label{th:Gallai} {\normalfont (\cite{Gallai,GySi})}
In any Gallai-coloring of a complete graph, the vertex set can be partitioned into at least two nonempty parts such that there is only one color on the edges between every pair of parts, and there are at most two colors between the parts in total.
\end{theorem}

We call a vertex partition as given by the statement in Theorem~\ref{th:Gallai} a {\it Gallai partition}. Below we listed some known exact values of Gallai-Ramsey numbers and Ramsey numbers.

\noindent\begin{theorem}\label{th:GRk3} {\normalfont (\cite{ChGr,GSSS})} For integers $k\geq 1$, we have
$$GR_k(K_3)=\left\{
                \begin{aligned}
                &5^{k/2}+1, & & \mbox{if $k$ is even},\\
                &2\cdot5^{(k-1)/2}+1, & & \mbox{if $k$ is odd}.
              \end{aligned}
           \right.$$
\end{theorem}

\noindent\begin{theorem}\label{th:ramsey} The following Ramsey numbers have been established:
\begin{itemize}
  \item[{\rm (1)}] {\normalfont (\cite{GrGl})} $R(K_3, K_3)=6$, $R(K_{4}, K_{4})=18$.

  \item[{\rm (2)}] {\normalfont (\cite{Clan})} $R(K_{4}+e, K_3)=9$.

  \item[{\rm (3)}] {\normalfont (\cite{HaMe})} $R(K_{4}+e, K_{4}+e)=18$.
\end{itemize}
\end{theorem}

For a graph $H$, let $\Delta (H)$ and $\chi(H)$ be the maximum degree and chromatic number of $H$, respectively. Given an edge-colored graph $F$ and an edge $e\in E(F)$, let $c_F(e)$ (or simply $c(e)$) be the color used on (i.e., assigned to) edge $e$. For $U$, $V\subseteq V(F)$ with $U\cap V= \emptyset$, we use $E(U, V)$ (resp., $C(U, V)$) to denote the set of edges between $U$ and $V$ (resp., the set of colors used on the edges between $U$ and $V$). If all the edges in $E(U, V)$ are colored by a single color, then we use $c(U, V)$ to denote this color. Let $F[U]$ be the subgraph of $F$ induced by $U\subseteq V(F)$, and $F-U$ be the subgraph of $F$ induced by $V(F)\setminus U$ (if $U\neq V(F)$). In the special case when $U=\{u\}$, we simply write $E(u, V)$, $C(u, V)$, $c(u, V)$ and $F-u$ for $E(\{u\}, V)$, $C(\{u\}, V)$, $c(\{u\}, V)$ and $F-\{u\}$, respectively. Let $C(F[U])$ (or simply, $C(U)$) and $C(F-U)$ denote the set of colors used on $E(F[U])$ and $E(F-U)$, respectively. For two graphs $F_1$ and $F_2$, let $F_1\cup F_2$ be the disjoint union of $F_1$ and $F_2$.

Next, we define the blow-up of an edge-colored complete graph which will be used in our proofs of Theorems \ref{th:multiplicity1} and \ref{th:GRk4+e}. Let $G$ be an edge-colored complete graph with vertex set $\{v_1, v_2, \ldots, v_n\}$, and $H_1, H_2, \ldots, H_n$ be $n$ pairwise disjoint edge-colored complete graphs. The {\it blow-up} $G(H_1, H_2, \ldots, H_n)$ of $G$ is an edge-colored complete graph with vertex set $\bigcup^{n}_{i=1}V(H_i)$ such that
$$c_{G(H_1, H_2, \ldots, H_n)}(xy)=
\left\{
   \begin{aligned}
    &c_G(v_iv_j), & & \mbox{if $x\in V(H_i)$ and $y\in V(H_j)$ for some $1\leq i\neq j\leq n$},\\
    &c_{H_i}(xy), & & \mbox{if $x, y\in V(H_i)$ for some $i\in [n]$}.
   \end{aligned}
   \right.$$
If $H_1=H_2=\cdots=H_n=H$, we will write $G(n\cdot H)$ for $G(H, H, \ldots, H)$. If $H_1=\cdots=H_s=H'$ and $H_{s+1}=\cdots=H_n=H''$ for some $1\leq s < n$, we will write $G(s\cdot H', (n-s)\cdot H'')$ for $G(H', \ldots, H', H'', \ldots, H'')$. Similarly, we will use the abbreviation $G(s\cdot H', t\cdot H'', (n-s-t)\cdot H''')$.

In the following, we will introduce the Regularity Lemma, Embedding Lemma and Slicing Lemma that will be used in our proof of Theorem \ref{th:RMk3k3}. Given a graph $F$ and two disjoint nonempty sets $X, Y\subseteq V(F)$, the {\it density} of $(X, Y)$ is defined to be $$d(X,Y):=\frac{|E(X,Y)|}{|X| |Y|}.$$ We say that $(X, Y)$ is {\it $\varepsilon$-regular} if for any $X'\subseteq X$ and $Y'\subseteq Y$ with $|X'|\geq \varepsilon |X|$ and $|Y'|\geq \varepsilon |Y|$, we have $|d(X', Y')-d(X, Y)|\leq \varepsilon$. For a positive real number $d$, we say that an $\varepsilon$-regular pair $(X, Y)$ is {\it $(\varepsilon, d)$-regular} if $d(X,Y)\geq d$.

\noindent\begin{lemma}\label{le:regularity} (Multicolor Regularity Lemma) {\normalfont (see e.g. \cite{KoSi,LiPS,Sze})} For any real $\varepsilon >0$ and positive integers $k$ and $m$, there exist $n'$ and $M$, such that every $k$-edge-colored graph $F$ with $n\geq n'$ vertices admits a partition $V_1, V_2, \ldots, V_t$ of $V(F)$ satisfying
\begin{itemize}
  \item[{\rm (i)}] $m\leq t\leq M$;
  \item[{\rm (ii)}] for all $1\leq i<j \leq t$, we have $||V_i|-|V_j||\leq 1$; and
  \item[{\rm (iii)}] for all but at most $\varepsilon \binom{t}{2}$ pairs $(i,j)$, the pair $(V_i, V_j)$ is $\varepsilon$-regular for each color.
\end{itemize}
\end{lemma}

We call the partition as given in Lemma \ref{le:regularity} a {\it multicolored $\varepsilon$-regular partition}. Given $\varepsilon, d >0$, a $k$-edge-colored graph $F$ and a partition $V_1, V_2, \ldots, V_t$ of $V(F)$, we define the {\it reduced graph} $R=R(d)$ as follows: $V(R)=\{1,2, \ldots, t\}$ and $i$ and $j$ are adjacent in $R$ if $(V_i, V_j)$ is $\varepsilon$-regular for each color and there exists a color with density at least $d$ in $E(V_i,V_j)$. Moreover, we define the {\it multicolored reduced graph} $R^c=R^c(d)$ as follows: $V(R^c)=V(R)$, $E(R^c)=E(R)$, and for each edge $ij \in E(R^c)$, $ij$ is assigned an arbitrary color $c_0$ such that $(V_i,V_j)$ has density at least $d$ with respect to the subgraph of $F$ induced by the edges of color $c_0$.

Given two graphs $G$ and $H$, we say that $G$ is a {\it homomorphic copy} of $H$ if there is a map $\varphi :V(H)\rightarrow V(G)$ such that $\varphi(u)\varphi(v)\in E(G)$ for each edge $uv\in E(H)$. Note that $K_{s}$ is a homomorphic copy of $H$ if and only if $s\geq \chi(H)$.
We will use the following consequence of the Embedding Lemma. Lemma \ref{le:embedding} below is in fact a corollary of Lemma 2.4 in \cite{HoLO'}.

\noindent\begin{lemma}\label{le:embedding} (Multicolor Embedding Lemma) {\normalfont (see e.g. \cite{HoLO',HoLO,KoSi})} For every $d>0$, any positive integer $k$ and any graph $G$, there exist $\varepsilon=\varepsilon(k, d, G)>0$ and a positive integer $n_0=n_0(k, d, G)$ with the following property. Suppose that $F$ is a $k$-edge-colored graph on $n\geq n_0$ vertices with a multicolored $\varepsilon$-regular partition $V_1, V_2, \ldots, V_t$ which defines the multicolored reduced graph $R^c=R^c(d)$. If $R^c$ contains a monochromatic homomorphic copy of $G$, then $F$ contains a monochromatic copy of $G$. If $R^c$ contains a rainbow copy of $G$, then $F$ contains a rainbow copy of $G$.
\end{lemma}

\noindent\begin{lemma}\label{le:slicing} (Slicing Lemma) {\normalfont (see e.g. \cite{KoSi,LiPS})} Let $0<\varepsilon, \alpha, d <1$ with $\varepsilon\leq \min\left\{d, \alpha, 1/2\right\}$. If a pair $(X,Y)$ is $(\varepsilon, d)$-regular, then for any $X'\subseteq X$ and $Y'\subseteq Y$ with $|X'|\geq \alpha |X|$ and $|Y'|\geq \alpha |Y|$, we have that $(X', Y')$ is an $(\varepsilon', d-\varepsilon)$-regular pair, where $\varepsilon':=\max\left\{2\varepsilon, \varepsilon/\alpha\right\}$.
\end{lemma}

Finally, we consider the Tur\'{a}n number. It is well-known that $ex(n, K_{r+1})=t(n,r)=\left(1-1/r\right)\binom{n}{2}+o(n^2)$. In fact, if $n\equiv p \pmod r$ where $0\leq p\leq r-1$, then $t(n,r)=\left(1-1/r\right)n^2/2+(p-r)p/(2r)$. Thus $\left(1-1/r\right)n^2/2-r/8\leq t(n,r)\leq \left(1-1/r\right)n^2/2$. We will use this more precise bound in our proofs of the main results.

\section{On edges not contained in a rainbow triangle or monochromatic copy of $H$}
\label{sec:ch-proof-RMk3k3}

For the proof of Theorem \ref{th:RMk3k3}, we first define the following variant of the Gallai-Ramsey number. Given a set $V$ and an integer $k \leq |V|$, let $\binom{V}{\leq k}$ (resp., $\binom{V}{k}$) be the set of all nonempty subsets of $V$ of size at most $k$ (resp., size $k$).

\noindent\begin{definition}\label{de:gr*} For a graph $H$ and an integer $k\geq 2$, let $GR^{\ast}_k(H)$ be the minimum integer $n^{\ast}$ such that for every coloring $c: \binom{[n^{\ast}]}{\leq 2} \rightarrow [k]$, at least one of the following statements holds:
\begin{itemize}
\item[{\rm(1$^{\ast}$)}] the restriction of $c$ to $\binom{[n^{\ast}]}{2}$ contains either a rainbow triangle or a monochromatic homomorphic copy of $H$;

\item[{\rm(2$^{\ast}$)}] for some $1\leq i < j\leq n^{\ast}$, we have $c(\{i,j\})= c(\{i\})$ or $c(\{i,j\})= c(\{j\})$.
\end{itemize}
\end{definition}

In other words, $GR^{\ast}_k(H)-1$ is the maximum integer $n^{\ast \ast}$ such that for the complete graph $K_{n^{\ast \ast}}$ with vertex set $[n^{\ast \ast}]$, there exists a coloring $c: \binom{[n^{\ast \ast}]}{\leq 2} \rightarrow [k]$ satisfying
\begin{itemize}
\item[{\rm(1$^{\ast \ast}$)}] the restriction of $c$ to $\binom{[n^{\ast \ast}]}{2}$ is a Gallai-$k$-coloring without a monochromatic homomorphic copy of $H$; and

\item[{\rm(2$^{\ast \ast}$)}] for any $1\leq i < j\leq n^{\ast}$, we have $c(\{i,j\})\neq c(\{i\})$ and $c(\{i,j\})\neq c(\{j\})$.
\end{itemize}

For a set $\mathscr{H}$ of graphs, let $GR_k(\mathscr{H})$ denote the minimum integer $n$ such that every Gallai-$k$-coloring of $K_n$ contains a monochromatic copy of $H$ for some $H\in\mathscr{H}$.

\noindent\begin{lemma}\label{le:gr*bound} For a graph $H$, let $\mathscr{H}$ be the set of all homomorphic copies of $H$. Then
\begin{itemize}
\item[{\rm(1)}] $GR^{\ast}_k(H)\geq GR_{k-1}(\mathscr{H})$,

\item[{\rm(2)}] $f_k(n,H)\geq t(n, GR_{k-1}(\mathscr{H})-1)$,

\item[{\rm(3)}] if there exists a coloring $c$ satisfying conditions {\rm (1$^{\ast \ast}$)} and {\rm (2$^{\ast \ast}$)} such that all elements of $\binom{[GR^{\ast}_k(H)-1]}{1}$ use a single color, then $f_k(n,H)\geq t(n, GR^{\ast}_k(H)-1)$.
\end{itemize}
\end{lemma}

\begin{proof}
Let $n^{\ast}_k:=GR_{k-1}(\mathscr{H})$. We first prove (1). Let $F$ be a Gallai-$(k-1)$-coloring of $K_{n^{\ast}_k-1}$ without a monochromatic copy of $H'$ for any $H'\in\mathscr{H}$. We color the vertices of $F$ with the $k$th color and then we obtain a $k$-coloring of $\binom{[n^{\ast}_k-1]}{\leq 2}$ satisfying conditions (1$^{\ast \ast}$) and (2$^{\ast \ast}$). Thus $GR^{\ast}_k(H)\geq n^{\ast}_k= GR_{k-1}(\mathscr{H})$.

Next, we give the proof of (2). Let $G$ be a Gallai-$(k-1)$-coloring of $K_{n^{\ast}_k-1}$ without a monochromatic copy of $H'$ for any $H'\in\mathscr{H}$. Let $V(G)=\{1, 2, \ldots, n^{\ast}_k-1\}$ and let $G'$ be the Tur\'{a}n graph $T_{n^{\ast}_k-1}(n)$
with parts $V_1, \ldots, V_{n^{\ast}_k-1}$. We color the edges of $G'$ such that for any $1\leq i<j\leq n^{\ast}_k-1$, we have $c_{G'}(V_i, V_j)=c_{G}(ij)$. Let $G''$ be a $k$-edge-coloring of $K_{n}$ obtained by coloring the edges within each part using color $k$ from the above $(k-1)$-edge-coloring of $G'$. We claim that all the edges between the $n^{\ast}_k-1$ parts are neither contained in a rainbow copy of $K_3$ nor in a monochromatic copy of $H$ in $G''$. Indeed, note that there is no rainbow copy of $K_3$ using color $k$. Thus if $G''$ contains a rainbow copy of $K_3$, then $G$ is not a Gallai-coloring, a contradiction. If there is an edge $e$ between these $n^{\ast}_k-1$ parts such that $e$ is contained in a monochromatic copy of $H$, then $G$ contains a monochromatic homomorphic copy of $H$, a contradiction. Thus $f_{k}(n, H)\geq \left|E\left(G'\right)\right|=t(n, n^{\ast}_k-1)$.

Finally, we prove (3). Let $n_k:=GR^{\ast}_k(H)-1$. Let $c$ be a coloring as in the statement of the lemma, and we may assume that all elements of $\binom{[n_k]}{1}$ are colored by color 1. Note that the restriction of $c$ to $\binom{[n_k]}{2}$ is a Gallai-$(k-1)$-coloring without a monochromatic homomorphic copy of $H$. Let $W$ be the Tur\'{a}n graph $T_{n_k}(n)$ with parts $V_1, \ldots, V_{n_k}$. We color the edges of $W$ such that $c_{W}(V_i, V_j)=c(\{i,j\})$ for any $1\leq i<j\leq n_k$. Let $W'$ be a $k$-edge-coloring of $K_{n}$ obtained by coloring the edges within each part using color $1$ from the above $(k-1)$-edge-coloring of $W$. It is easy to check that all the edges between the $n_k$ parts are neither contained in a rainbow copy of $K_3$ nor in a monochromatic copy of $H$ in $W'$. Thus $f_{k}(n, H)\geq \left|E\left(W\right)\right|=t(n, n_k)$.
\end{proof}

Note that we have $GR^{\ast}_k(H)= GR_{k-1}(\mathscr{H})=2$ whenever $H$ is a bipartite graph, where $\mathscr{H}$ is the set of all homomorphic copies of $H$. A natural question is  for which non-bipartite graph $H$ it holds that $GR^{\ast}_k(H)= GR_{k-1}(\mathscr{H})$? We can verify that $K_3$ is such a graph.

\noindent\begin{lemma}\label{le:gr*} Let $\mathscr{H}(K_3)$ be the set of all homomorphic copies of $K_3$. For integers $k\geq 2$, we have $GR^{\ast}_k(K_3) = GR_{k-1}(\mathscr{H}(K_3)) = GR_{k-1}(K_3)$.
\end{lemma}

\begin{proof}
For every graph $H'\in \mathscr{H}(K_3)$, we have that $H'$ contains $K_3$ as a subgraph by the definition. Thus $GR_{k-1}(\mathscr{H}(K_3))\geq GR_{k-1}(K_3)$. By Lemma \ref{le:gr*bound} (1), we have $GR^{\ast}_k(K_3) \geq GR_{k-1}(\mathscr{H}(K_3)) \geq GR_{k-1}(K_3)$.

For $k\geq 2$, let $n^{\ast}_k:=GR_{k-1}(K_3)$, and we will prove $GR^{\ast}_k(K_3)\leq n^{\ast}_k$ by induction on $k$.
When $k=2$, we have $GR^{\ast}_2(K_3)=3=n^{\ast}_2$ clearly. Suppose that for all $2\leq k'\leq k-1$, we have $GR^{\ast}_{k'}(K_3)\leq n^{\ast}_{k'}$. We will prove it for $k'=k$. Let $n$ be the maximum integer such that there is a coloring $c: \binom{[n]}{\leq 2} \rightarrow [k]$ satisfying conditions (1$^{\ast \ast}$) and (2$^{\ast \ast}$). It suffices to show that $n\leq n^{\ast}_k-1$. By Theorem \ref{th:Gallai}, there is a Gallai partition $V_1, V_2, \ldots, V_m$ ($m\geq 2$) of $[n]$. Note that $K_3\in \mathscr{H}(K_3)$. For avoiding a monochromatic copy of $K_3$, we have $m\leq 5$. We choose such a partition so that $m$ is minimum. Let $R$ be an edge-coloring of a complete graph with $V(R)=\{v_1, v_2, \ldots, v_m\}$ and $c(v_iv_j)=c(V_i, V_j)$ for any $i\neq j$. If $m=5$ (resp., $m=4$), then $R$ is the unique 2-edge-coloring of $K_5$ without a monochromatic copy of $K_3$, i.e., each color forms a cycle of length 5 (resp., $R$ is one of the two 2-edge-colorings of $K_4$ without a monochromatic copy of $K_3$, i.e., each color forms a path of length 3, or one color forms a cycle of length 4 and the other color forms a matching with two edges). Then there is no edge using color 1 or 2 within each part $V_i$ for avoiding a monochromatic copy of $K_3$, and there is no vertex using color 1 or 2 within each part $V_i$ by condition (2$^{\ast \ast}$). Thus if $k=3$, then $n\leq 5=n^{\ast}_3-1$, and if $k\geq 4$, then $n\leq 5(GR^{\ast}_{k-2}(K_3)-1)\leq 5(GR_{k-3}(K_3)-1)\leq n^{\ast}_k-1$ by the induction hypothesis. If $m=3$, then at least two of the colors $c(V_1, V_2)$, $c(V_1, V_3)$ and $c(V_2, V_3)$ are the same color, say $c(V_1, V_2)=c(V_1, V_3)$. This implies that $V_1$ and $V_2\cup V_3$ form a Gallai partition with exactly two parts, contradicting the minimality of $m$. If $m=2$, then we may assume $c(V_1, V_2)=1$. Then color 1 cannot be used on $\binom{V_1}{\leq 2}$ and $\binom{V_2}{\leq 2}$. Thus $n\leq 2(GR^{\ast}_{k-1}(K_3)-1)\leq2(GR_{k-2}(K_3)-1)\leq n^{\ast}_k-1$ by the induction hypothesis.
\end{proof}

By Lemma \ref{le:gr*}, we have $GR^{\ast}_k(K_3) = GR_{k-1}(\mathscr{H}(K_3))$. As in the proof of Lemma \ref{le:gr*bound} (1), we can construct an extremal coloring $\binom{[GR^{\ast}_k(K_3)-1]}{\leq 2} \rightarrow [k]$ satisfying conditions (1$^{\ast \ast}$) and (2$^{\ast \ast}$) in which we assign a single color to all elements of $\binom{[GR^{\ast}_k(K_3)-1]}{1}$. It is worth noticing that not all the extremal colorings assign a single color to all singletons. For example, Figure \ref{fig:figure1} gives an extremal coloring of $GR^{\ast}_4(K_3)$ with two colors on singletons.

\begin{figure}[htbp]
\begin{center}
\begin{tikzpicture}[scale=0.06,auto,swap]
\tikzstyle{redvertex}=[circle,draw=red,fill=red]
\tikzstyle{bluevertex}=[circle,draw=blue,fill=blue]
 \node[redvertex,scale=0.5]  (v11) at (28.71,-3.57) {};  \node[bluevertex,scale=0.5]  (v12) at (42.4,-3.57) {};
 \node[redvertex,scale=0.5]  (v21) at (6.04,-20.85) {};  \node[redvertex,scale=0.5]  (v22) at (2.02,-33.21) {};
 \node[bluevertex,scale=0.5]  (v31) at (11.42,-60.22) {};  \node[bluevertex,scale=0.5]  (v32) at (21.94,-67.86) {};
 \node[bluevertex,scale=0.5]  (v41) at (50.52,-67.86) {};  \node[bluevertex,scale=0.5]  (v42) at (61.04,-60.22) {};
 \node[redvertex,scale=0.5]  (v51) at (69.3,-33.21) {};  \node[redvertex,scale=0.5]  (v52) at (65.29,-20.85) {};

 \draw [black,thick] (v11) -- (v12);
 \draw [green,thick] (v21) -- (v22);
 \draw [green,thick] (v31) -- (v32);
 \draw [red,thick] (v41) -- (v42);
 \draw [blue,thick]  (v51) -- (v52);

 \draw [blue,thick] (v11) -- (v21); \draw [blue,thick] (v11) -- (v22); \draw [black,thick] (v12) -- (v21); \draw [black,thick] (v12) -- (v22);

 \draw [black,thick] (v21) -- (v31); \draw [black,thick] (v21) -- (v32); \draw [black,thick] (v22) -- (v31); \draw [black,thick] (v22) -- (v32);

 \draw [green,thick] (v31) -- (v41); \draw [red,thick] (v31) -- (v42); \draw [red,thick] (v32) -- (v41); \draw [green,thick] (v32) -- (v42);

 \draw [black,thick] (v41) -- (v51); \draw [black,thick] (v41) -- (v52); \draw [black,thick] (v42) -- (v51); \draw [black,thick] (v42) -- (v52);

 \draw [green,thick] (v51) -- (v11); \draw [black,thick] (v51) -- (v12); \draw [green,thick] (v52) -- (v11); \draw [black,thick] (v52) -- (v12);

 \draw [black,thick] (v21) -- (v41); \draw [black,thick] (v21) -- (v42); \draw [black,thick] (v22) -- (v41); \draw [black,thick] (v22) -- (v42);

 \draw [black,thick] (v31) -- (v51); \draw [black,thick] (v31) -- (v52); \draw [black,thick] (v32) -- (v51); \draw [black,thick] (v32) -- (v52);

 \draw [blue,thick] (v21) -- (v51); \draw [green,thick] (v21) -- (v52); \draw [green,thick] (v22) -- (v51); \draw [blue,thick] (v22) -- (v52);

 \draw [black,thick] (v11) -- (v31); \draw [black,thick] (v11) -- (v32); \draw [red,thick] (v12) -- (v31); \draw [red,thick] (v12) -- (v32);

 \draw [black,thick] (v11) -- (v41); \draw [black,thick] (v11) -- (v42); \draw [green,thick] (v12) -- (v41); \draw [green,thick] (v12) -- (v42);
\end{tikzpicture}
\caption{An extremal coloring of $GR^{\ast}_4(K_3)$ with two colors on singletons.}
\label{fig:figure1}
\end{center}
\end{figure}
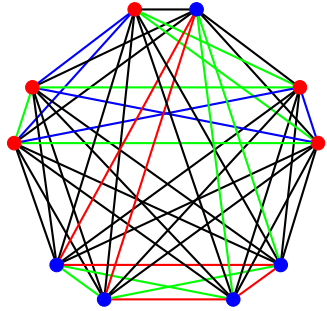

Now we have all the ingredients for our proof of Theorem \ref{th:RMk3k3}.

\begin{proof}[Proof of Theorem \ref{th:RMk3k3}]
The lower bound follows from Lemmas \ref{le:gr*bound} (2) and \ref{le:gr*}. Next, we will prove that $f_{k}(n, K_3)< t(n, GR_{k-1}(K_3)-1)+\delta n^2$. Let nim$_k(n, K_3)$ be the maximum number of edges not contained in any monochromatic copy of $K_3$ over all $k$-edge-colorings of $K_n$. Note that $f_{k}(n, K_3)$ $\leq \mbox{nim}_k(n,$ $K_3)$. For sufficiently large $n$, since nim$_2(n, K_3)=t(n, 2)$ (proven in \cite{KeSu}) and nim$_3(n, K_3)=t(n, 5)$ (proven in \cite{LiPS}), we have $f_{k}(n, K_3)= t(n, GR_{k-1}(K_3)-1)$ for $k\in\{2,3\}$. In the following, we may assume $k\geq 4$.

Let $N_k:=GR_{k-1}(K_3)$. We choose $d$ such that $d\leq \delta /k$. Let $\varepsilon_1=\varepsilon_1(k, d/2, K_3)$ and $n_1=n_1(k, d/2, K_3)$ (resp., $\varepsilon_2=\varepsilon_2(k, d, K_3)$ and $n_2=n_2(k, d, K_3)$) be the values obtained by applying Lemma \ref{le:embedding}. Let $n'_1$ and $M_1$ be the values obtained by applying Lemma \ref{le:regularity} with $\varepsilon_1$ and $1/ \varepsilon_1$. Then we choose $\varepsilon$ such that $\varepsilon\leq \min\left\{\delta/4, \varepsilon_1/M_1, \varepsilon_2, d/2\right\}$. Let $n'$ and $M$ be the values obtained by applying Lemma \ref{le:regularity} with $\varepsilon$ and $1/\varepsilon$. Let $n_0= \max\left\{n', n'_1M, \sqrt{(N_k-1)/(2\delta)}, MM_1n_1/3, n_2\right\}$ and $n\geq n_0$.

Let $F$ be a $k$-edge-coloring of $K_n$, and $F'$ be the spanning subgraph of $F$ with $E(F')=\{e\in E(F)\colon\, e$ is not contained in any rainbow or monochromatic copy of $K_3\}$. For a contradiction, suppose $|E(F')|\geq t(n, N_k-1)+\delta n^2$. Let $V_1, V_2, \ldots, V_t$ be a partition of $V(F')$ obtained by applying Lemma \ref{le:regularity} to $F'$ with $\varepsilon$ and $1/ \varepsilon$, where $1/ \varepsilon \leq t \leq M$. Let $R=R(d)$ be the reduced graph. Since there are at most $\binom{n/t}{2}$ edges within a part, at most $\left(n/t\right)^2$ edges between any two parts, and less than $kd\left(n/t\right)^2$ edges between a pair of parts with density less than $d$ for each color,
we have
\begin{align*}
 |E(R)| > & \ \frac{t(n,N_k-1)+\delta n^2- t\binom{\frac{n}{t}}{2}- \varepsilon\binom{t}{2}\left(\frac{n}{t}\right)^2- kd\left(\frac{n}{t}\right)^2\binom{t}{2}}{\left(\frac{n}{t}\right)^2}\\
   > & \ \frac{t^2\left(\left(1-\frac{1}{N_k-1}\right)\frac{n^2}{2}-\frac{N_k-1}{8}+\delta n^2- \left(\frac{1}{t}+\varepsilon+kd\right)\frac{n^2}{2}\right)}{n^2}\\
   = & \ \left(1-\frac{1}{N_k-1}+2\delta-\frac{N_k-1}{4n^2}-\frac{1}{t}-\varepsilon-kd\right) \frac{t^2}{2}\\
   \geq & \ \left(1-\frac{1}{N_k-1}\right) \frac{t^2}{2},
\end{align*}
where the last inequality is by the choices of $n$, $d$ and $\varepsilon$. Thus $|E(R)| \geq t(t, N_k-1)+1$, so $R$ contains a copy $R'$ of $K_{N_k}$. Without loss of generality, let $V(R')=\{1, 2, \ldots, N_k\}$. Then for any $1\leq i<j\leq N_k$, we have that $(V_i, V_j)$ is $\varepsilon$-regular for each color, and there exists a color $c_{ij}$ with density at least $d$ in $E(V_i, V_j)$.

For each $i\in [N_k]$, we have $|V_i|=n/t\geq (n'_1M)/M=n'_1$. Thus we can apply Lemma \ref{le:regularity} with $\varepsilon_1$ and $1/ \varepsilon_1$ to $F[V_i]$ (note that here we consider $F[V_i]$, not only $F'[V_i]$). Then there exist two subsets $V_{i,1}, V_{i,2}\subseteq V_i$ with $|V_{i,1}|=|V_{i,2}|\geq n'_1/M_1$ such that $(V_{i,1}, V_{i,2})$ is an $(\varepsilon_1, 1/k)$-regular pair for some color $c_i\in [k]$. From the choice of $d$, we have $1/k\geq d/2$, so $(V_{i,1}, V_{i,2})$ is an $(\varepsilon_1, d/2)$-regular pair for color $c_i$. We define a coloring $\varphi: \binom{V(R')}{\leq 2}\rightarrow [k]$ such that $\varphi(\{i\})=c_i$ and $\varphi(\{i, j\})=c_{ij}$. Note that there might be more than one choice for $\varphi(\{i\})$ and $\varphi(\{i, j\})$, and we may choose an arbitrary one from these choices. By Lemma \ref{le:gr*}, we have $|V(R')|=N_k=GR_{k-1}(K_3)=GR^{\ast}_k(K_3)$. Thus at least one of the following statements holds:
\begin{itemize}
  \item[{\rm (1)}] $R'$ contains a rainbow copy of $K_3$;

  \item[{\rm (2)}] $R'$ contains a monochromatic homomorphic copy of $K_3$;

  \item[{\rm (3)}] $\varphi(\{i, j\})=\varphi(\{i\})$ for some $1\leq i\neq j\leq N_k$.
\end{itemize}

If (1) or (2) holds, then there is a rainbow or monochromatic copy of $K_3$ in $F'$ by Lemma~\ref{le:embedding}, a contradiction. If (3) holds, then by applying Lemma \ref{le:slicing} with $\alpha=1/M_1$, we have that $(V_{j}, V_{i,1})$ and $(V_{j}, V_{i,2})$ are two $(\varepsilon M_1, d-\varepsilon)$-regular (and thus $(\varepsilon_1, d/2)$-regular) pairs for color $c_i$. Thus $(V_{i,1}, V_{i,2})$, $(V_{j}, V_{i,1})$ and $(V_{j}, V_{i,2})$ are three $(\varepsilon_1, d/2)$-regular pairs for color $c_i$. By Lemma \ref{le:embedding}, there is a monochromatic copy of $K_3$ which contains two edges of $F'$, a contradiction.
\end{proof}

By similar arguments as in the proof of Theorem \ref{th:RMk3k3}, we can prove the following result for a general graph $H$. We omit the details.

\noindent\begin{theorem}\label{th:RMk3H} For any $\delta>0$, there exists an $n_0$ such that for all $n\geq n_0$ and any graph $H$, we have $t\left(n, GR_{k-1}(\mathscr{H})- 1\right)\leq f_{k}(n, H)<t\left(n, GR^{\ast}_{k}(H)-1\right)+\delta n^2$, where $\mathscr{H}$ is the set of all homomorphic copies of $H$.
\end{theorem}

\section{The Ramsey multiplicity problem for Gallai-colorings}
\label{sec:ch-proof-multiplicity}

We first prove the upper bound in Theorem \ref{th:multiplicity1}, by construction. Let $G_2$ be a 2-edge-colored $K_5$ using colors $1$ and $2$ which contains no monochromatic copy of $K_3$, i.e., colors $1$ and $2$ induce two monochromatic copies of $C_5$. Suppose that $2i < k-2$ and we have constructed a Gallai-$2i$-coloring $G_{2i}$ of $K_{n_{2i}}$ without a monochromatic copy of $K_3$, where $n_{2i}:=5^{i}$. Let $G'$ be a 2-edge-colored $K_5$ using colors $2i+1$ and $2i+2$ which contains no monochromatic copy of $K_3$. Let $G_{2i+2}=G'(5\cdot G_{2i})$, i.e., $G_{2i+2}$ is a blow-up of $G'$. This way, when $k$ is odd (resp., $k$ is even), we obtain a Gallai-$(k-1)$-coloring $G_{k-1}$ of $K_{n_{k-1}}$ (resp., Gallai-$(k-2)$-coloring $G_{k-2}$ of $K_{n_{k-2}}$) without a monochromatic copy of $K_3$, where $n_{k-1}=5^{(k-1)/2}$ (resp., $n_{k-2}=5^{(k-2)/2}$). In the following, we will construct a Gallai-$k$-coloring $G_k$ from $G_{k-1}$ or $G_{k-2}$.

If $k$ is odd, then let $A$ be a monochromatic copy of $K_{m}$ using color $k$, and let $B$ be a monochromatic copy of $K_{m+1}$ using color $k$. Let $G_k=G_{k-1}(r\cdot B, (5^{(k-1)/2}-r) \cdot A)$. Then $G_k$ is a Gallai-$k$-coloring of $K_n$ with $r\binom{m+1}{3}+\left(5^{(k-1)/2}-r\right)\binom{m}{3}$ monochromatic copies of $K_3$ (here we define $\binom{1}{3}=\binom{2}{3}=0$ for the sake of notation). If $k$ is even, then let $C$ be a 2-edge-coloring (using colors $k-1$ and $k$) of $K_{m}$ with $M_2(K_3, m)$ monochromatic copies of $K_3$, and let $D$ be a 2-edge-coloring (using colors $k-1$ and $k$) of $K_{m+1}$ with $M_2(K_3, m+1)$ monochromatic copies of $K_3$. Let $G_k=G_{k-2}(r\cdot D, (5^{(k-2)/2}-r) \cdot C)$. Then $G_k$ is a Gallai-$k$-coloring of $K_n$ with $rM_2(K_3, m+1)+\left(5^{(k-2)/2}-r\right)M_2(K_3, m)$ monochromatic copies of $K_3$. This completes the proof for the upper bound in Theorem~\ref{th:multiplicity1}.

It is worth noting that no matter whether $k$ is odd or even, the above extremal coloring is a blow-up of a complete graph of order $5^{\lfloor(k-1)/2\rfloor}$ with a special edge-coloring. Recall that we have $g_{3}(K_3,n)=r\binom{m+1}{3}+(5-r)\binom{m}{3}$. An interesting fact is that the above sharpness example for $k=3$ is the unique Gallai-3-coloring of $K_n$ achieving the minimum number of monochromatic copies of $K_3$, which can be derived from a result of \cite{CKPSTY}. But when $k$ is an even number, the extremal colorings achieving the upper bound are not unique. For example, let $F$ be a 2-edge-coloring (using colors $k-1$ and $k$) of $K_{m+2}$ with $M_2(K_3, m+2)$ monochromatic copies of $K_3$. Since $M_2(K_3, m)+M_2(K_3, m+2)=2M_2(K_3, m+1)$ for any odd number $m$ by Theorem \ref{th:M2}, we can also construct $G_{k}$ such that $G_k=G_{k-2}(1 \cdot F, (r-2)\cdot D, (5^{(k-2)/2}-r+1) \cdot C)$. However, it is still a blow-up of a complete graph of order $5^{\lfloor(k-1)/2\rfloor}$ with a special edge-coloring.

Before presenting our proof for the lower bound in Theorem \ref{th:multiplicity1}, we first provide the exact value of $g_{k}(K_3,GR_k(K_3))$.

\noindent\begin{theorem}\label{th:multiplicity2}
$g_{k}(K_3,GR_k(K_3))=1$ if $k$ is odd, and $g_{k}(K_3,GR_k(K_3))=2$ if $k$ is even.
\end{theorem}

\begin{proof}
By the definition of the Gallai-Ramsey number, we have $g_{k}(K_3,GR_k(K_3))\geq 1$. Moreover, it follows from the above extremal coloring that $g_{k}(K_3,GR_k(K_3))\leq 1$ if $k$ is odd, and $g_{k}(K_3,GR_k(K_3))\leq2$ if $k$ is even. Thus it suffices to prove that $g_{k}(K_3,GR_k(K_3))\geq2$ when $k$ is even. We will prove this by induction on $k$. For $k=2$, the statement is trivial since $M_2(K_3,6)=2$. We may assume that the statement holds for all even $k'\leq k-2$ and we will prove it for $k$ ($k\geq 4$).

Let $F$ be a Gallai-$k$-coloring of $K_{GR_k(K_3)}$ and suppose (for a contradiction) that $F$ contains only one monochromatic copy of $K_3$. Using Theorem \ref{th:Gallai}, let $V_1, V_2, \ldots, V_t$ ($t\geq 2$) be a Gallai partition of $V(F)$. We choose such a partition so that $t$ is minimum. We may assume that colors 1 and 2 are the two colors used between these parts. Let $R$ be a 2-edge-coloring of $K_t$ with $V(R)=\{v_1, v_2, \ldots, v_t\}$ and $c(v_iv_j)=c(V_i, V_j)$ for any $1\leq i<j\leq t$. Since $M_2(K_3,6)=2$, we have $t\leq 5$; otherwise $F$ contains at least two monochromatic copies of $K_3$.

If $2\leq t\leq 3$, then we may assume that $t=2$ by the minimality of $t$ (since every graph admitting a Gallai partition with three parts also admits a Gallai partition with two parts). Without loss of generality, let $c(V_1, V_2)=1$ and $|V_1|\geq |V_2|$. First, assume $1\notin C(V_1)$. Then $F[V_1]$ is a Gallai-$(k-1)$-coloring. Note that $|V_1|\geq |V(F)|/2\geq (5^{k/2}+1)/2> 2\cdot 5^{(k-2)/2}+2$. Since $k$ is even, we have $GR_{k-1}(K_3)=2\cdot5^{(k-2)/2}+1$. Thus there is a monochromatic copy of $K_3$ in $F[V_1]$. Let $v$ be a vertex of this $K_3$. Since $|V_1\setminus \{v\}|\geq 2\cdot5^{(k-2)/2}+1$, there is a monochromatic copy of $K_3$ in $F[V_1\setminus \{v\}]$. So there exist two monochromatic copies of $K_3$ in $F[V_1]$, a contradiction. We conclude that $1\in C(V_1)$. In order to avoid two monochromatic copies of $K_3$, we have $|V_2|=1$ and there is at most one edge with color 1 in $F[V_1]$. Thus there is a Gallai-$(k-1)$-coloring of $K_{|V_1|-1}$. Since $|V_1|-1\geq GR_{k-1}(K_3)$, there is a monochromatic copy of $K_3$ in $F[V_1]$. Then there exist two monochromatic copies of $K_3$ in $F$, another contradiction. This solves the case $2\leq t\leq 3$.

If $t=4$, then we first suppose that $R$ contains a monochromatic copy of $K_3$, say $c(V_1, V_2)=c(V_2, V_3)=c(V_3, V_1)=1$. Let $V'=V_1\cup V_2\cup V_3$. If $c(V_4, V')=2$, then $V_4$ and $V'$ form a Gallai partition with exactly two parts, contradicting the minimality of $t$. Thus $c(V_4, V_i)=1$ for some $i\in \{1, 2, 3\}$. But then $c(V_i, V(G)\setminus V_i)=1$, contradicting the minimality of $t$. Therefore, $R$ is one of the two 2-edge-colorings of $K_4$ without a monochromatic copy of $K_3$, that is, each color induces a path of length three, or one color induces a cycle of length four and  the other color induces a matching with two edges. In both cases we can derive that there is at most one edge with color 1 or 2 in $\bigcup^{4}_{j=1} F[V_j]$. By the induction hypothesis, we have $|V(F)|\leq 4(GR_{k-2}(K_3)-1)+1< GR_k(K_3)$, a contradiction.

The remaining case is $t=5$. Then there is no edge with color 1 or 2 in $\bigcup^{5}_{j=1} F[V_j]$; otherwise $F$ contains a 2-edge-coloring of $K_6$ which contains at least two monochromatic copies of $K_3$. Thus we have $|V(F)|\leq 5(GR_{k-2}(K_3)-1)< GR_k(K_3)$ by the induction hypothesis, a contradiction. This completes the proof of Theorem \ref{th:multiplicity2}.
\end{proof}

Now we have all ingredients to present our proof for the lower bound in Theorem~\ref{th:multiplicity1}. Let $s_0=1$ if $k$ is odd, and $s_0=2$ if $k$ is even. By Theorem~\ref{th:multiplicity2}, we have $g_{k}(K_3,GR_k(K_3))=s_0$. This implies that if $v_1, v_2, \ldots, v_{GR_k(K_3)}$ are any $GR_k(K_3)$ vertices of $K_n$, then $K_n[\{v_1, v_2, \ldots,$ $v_{GR_k(K_3)}\}]$ contains at least $s_0$ monochromatic copies of $K_3$. Since each monochromatic copy of $K_3$ is contained in $\binom{n-3}{GR_k(K_3)-3}$ distinct copies of $K_{GR_k(K_3)}$, there are at least
$$\left\lceil\frac{s_0\binom{n}{GR_k(K_3)}}{\binom{n-3}{GR_k(K_3)-3}}\right\rceil =\left\lceil\frac{s_0n(n-1)(n-2)}{GR_k(K_3)(GR_k(K_3)-1)(GR_k(K_3)-2)}\right\rceil$$
monochromatic copies of $K_3$ in any Gallai-$k$-coloring of $K_n$. This completes the proof of Theorem~\ref{th:multiplicity1}.

We obtain the following corollary.

\noindent\begin{corollary}\label{co:multiplicity1} If $k$ is odd and $0\leq t\leq 5^{(k-1)/2}-1$, then $g_{k}(K_3,GR_k(K_3)+t)=t+1$.
\end{corollary}

\begin{proof}
The upper bound follows from Theorem \ref{th:multiplicity1}. For the proof of the lower bound, we will use induction on $t$. The case $t=0$ follows from Theorem \ref{th:multiplicity2}. We may assume that $g_{k}(K_3,GR_k(K_3)+(t-1))=(t-1)+1=t$ holds and we will prove it for $t$ ($1\leq t\leq 5^{(k-1)/2}-1$). Let $n=GR_k(K_3)+t$. Note that each monochromatic copy of $K_3$ is contained in $\binom{n-3}{n-1-3}=n-3$ distinct copies of $K_{n-1}$, and there are $\binom{n}{n-1}=n$ distinct copies of $K_{n-1}$ in $K_n$. By the induction hypothesis, there are at least $\lceil tn/(n-3)\rceil=t+1$ monochromatic copies of $K_3$ in any Gallai-$k$-coloring of $K_n$.
\end{proof}

\section{The Gallai-Ramsey number for $K_4+e$}
\label{sec:ch-proof-GRk4+4}

For an integer $s$ with $0\leq s\leq k$, if $H_1=\cdots=H_s=K_4+e$ and $H_{s+1}=\cdots=H_{k}=K_3$, we will write $GR_k(s\cdot K_4+e, (k-s)\cdot K_3)$ for $GR(K_4+e, \ldots, K_4+e, K_3, \ldots, K_3)$. In this section, we will prove Theorem \ref{th:GRk4+e} in the following more general form. Theorem \ref{th:GRk4+e} follows from Theorem \ref{th:GRk4+egeneral} by choosing $s=k$.

\noindent\begin{theorem}\label{th:GRk4+egeneral} For integers $k\geq 1$ and $0\leq s\leq k$, we have
$$
GR_k(s\cdot K_4+e, (k-s)\cdot K_3) =
\begin{cases}
17^{s/2}\cdot5^{(k-s)/2}+1,  & \text{if $s$ is even and $k-s$ is even,} \\
2\cdot17^{s/2}\cdot5^{(k-s-1)/2}+1,  & \text{if $s$ is even and $k-s$ is odd,} \\
8\cdot17^{(s-1)/2}\cdot5^{(k-s-1)/2}+1,  & \text{if $s$ is odd and $k-s$ is odd,} \\
4\cdot 17^{(s-1)/2}\cdot5^{(k-s)/2}+1, & \text{if $s$ is odd and $k-s$ is even.}
\end{cases}
$$
\end{theorem}

\begin{proof}
For convenience, let
$$
g(k,s) :=
\begin{cases}
17^{s/2}\cdot5^{(k-s)/2},  & \text{if $s$ is even and $k-s$ is even,} \\
2\cdot17^{s/2}\cdot5^{(k-s-1)/2},  & \text{if $s$ is even and $k-s$ is odd,} \\
8\cdot17^{(s-1)/2}\cdot5^{(k-s-1)/2},  & \text{if $s$ is odd and $k-s$ is odd,} \\
4\cdot 17^{(s-1)/2}\cdot5^{(k-s)/2}, & \text{if $s$ is odd and $k-s$ is even.}
\end{cases}
$$

We first prove $GR_k(s\cdot K_4+e, (k-s)\cdot K_3)>g(k,s)$ by construction. Let $G_0$ be a single vertex and $G_1$ be a monochromatic copy of $K_4$ using color 1. If $s$ is even, then we will begin with $G_0$ and iteratively construct Gallai-colored graphs. If $s$ is odd, then we will begin with $G_1$ and iteratively construct Gallai-colored graphs. Suppose we have constructed $G_i$ for some $i< k$. Let $G'$ be a 2-edge-colored $K_5$ using colors $i+1$ and $i+2$ which contains no monochromatic copy of $K_3$, and $G''$ be a 2-edge-colored $K_{17}$ using colors $i+1$ and $i+2$ which contains no monochromatic copy of $K_4$. We construct $G_{i+2}$ or $G_{i+1}$ based on the following rules:
\begin{itemize}
\item[(1)] If $i\leq s-2$, then we construct $G_{i+2}$ such that $G_{i+2}=G''(17 \cdot G_i)$.
\item[(2)] If $s\leq i\leq k-2$, then we construct $G_{i+2}$ such that $G_{i+2}=G'(5 \cdot G_i)$.
\item[(3)] If $i=k-1$, then we construct $G_{i+1}$ by connecting two copies of $G_{i}$ with edges using color $k$.
\end{itemize}
Finally, we obtain a $g(k,s)$-vertex Gallai-$k$-colored graph $G_k$ containing neither a monochromatic copy of $K_4+e$ in any of the first $s$ colors nor a monochromatic copy of $K_3$ in any of the last $k-s$ colors.

In the following, we will prove $GR_k(s\cdot K_4+e, (k-s)\cdot K_3)\leq g(k,s)+1$ by induction on $k+s$. The case $k=1$ is trivial, the case $k=2$ follows from Theorem \ref{th:ramsey}, and the case $s=0$ follows from Theorem \ref{th:GRk3}. So we may assume that the result holds for all $k'+s'<k+s$ and we will prove it for $k+s$, where $k\geq 3$ and $1\leq s\leq k$.

Let $G$ be a Gallai-$k$-coloring of $K_{n}$, where $n=g(k,s)+1$. For a contradiction, suppose that $G$ contains neither a monochromatic copy of $K_4+e$ in any of the first $s$ colors nor a monochromatic copy of $K_3$ in any of the last $k-s$ colors. By Theorem \ref{th:Gallai}, let $V_1, V_2, \ldots, V_t$ ($t\geq 2$) be a Gallai partition of $V(G)$. We choose such a partition so that $t$ is minimum. We may assume that red and blue are the two colors used between these parts, where red and blue are two of the $k$ colors. Note that $n=g(k,s)+1\geq 21$ since $k\geq 3$ and $1\leq s\leq k$.

\noindent\begin{claim}\label{cl:GRk4+egeneral-1} $t\geq 4$.
\end{claim}

\begin{proof} If $t=3$, then at least two of the colors $c(V_1, V_2)$, $c(V_1, V_3)$ and $c(V_2, V_3)$ are the same color, say $c(V_1, V_2)=c(V_1, V_3)$. This implies that $V_1$ and $V(G)\setminus V_1$ form a Gallai partition with exactly two parts, contradicting the minimality of $t$. Hence, $t=2$, and we may assume that $c(V_1, V_2)$ is red without loss of generality.

If there is no red edge within both $V_1$ and $V_2$, then $G[V_1]$ and $G[V_2]$ are two Gallai-$(k-1)$-colorings. By the induction hypothesis, if red is one of the first $s$ colors, then we have
\begin{equation*}
\begin{aligned}
n &=|V_1|+|V_2|\leq 2\cdot g(k-1,s-1) \\
&=
\begin{cases}
2\cdot17^{(s-1)/2}\cdot5^{(k-s)/2},  & \text{if $s-1$ is even ($s$ is odd) and $k-s$ is even,} \\
2\cdot2\cdot17^{(s-1)/2}\cdot5^{(k-s-1)/2},  & \text{if $s-1$ is even ($s$ is odd) and $k-s$ is odd,} \\
2\cdot8\cdot17^{(s-2)/2}\cdot5^{(k-s-1)/2},  & \text{if $s-1$ is odd ($s$ is even) and $k-s$ is odd,} \\
2\cdot4\cdot 17^{(s-2)/2}\cdot5^{(k-s)/2}, & \text{if $s-1$ is odd ($s$ is even) and $k-s$ is even}
\end{cases} \\
&\leq g(k,s),
\end{aligned}
\end{equation*}
a contradiction. If red is one of the last $k-s$ colors, then we have
\begin{equation*}
\begin{aligned}
n &=|V_1|+|V_2|\leq 2\cdot g(k-1,s) \\
&=
\begin{cases}
2\cdot17^{s/2}\cdot5^{(k-s-1)/2},  & \text{if $s$ is even and $k-s-1$ is even ($k-s$ is odd),} \\
2\cdot2\cdot17^{s/2}\cdot5^{(k-s-2)/2},  & \text{if $s$ is even and $k-s-1$ is odd ($k-s$ is even),} \\
2\cdot8\cdot17^{(s-1)/2}\cdot5^{(k-s-2)/2},  & \text{if $s$ is odd and $k-s-1$ is odd ($k-s$ is even),} \\
2\cdot4\cdot 17^{(s-1)/2}\cdot5^{(k-s-1)/2}, & \text{if $s$ is odd and $k-s-1$ is even ($k-s$ is odd)}
\end{cases} \\
&\leq g(k,s),
\end{aligned}
\end{equation*}
a contradiction.

Thus we may assume that $G[V_1]$ contains a red edge, so red is one of the first $s$ colors. In order to avoid a red copy of $K_4+e$, there is no red edge within $V_2$ and there is no red copy of $K_3$ within $V_1$ (recall that $n\geq 21$). By the induction hypothesis, we have
\begin{equation*}
\begin{aligned}
n &=|V_1|+|V_2|\leq g(k,s-1)+g(k-1, s-1) \\
&=
\begin{cases}
8\cdot17^{(s-2)/2}\cdot5^{(k-s)/2}+4\cdot 17^{(s-2)/2}\cdot5^{(k-s)/2},  & \text{if $s$ is even and $k-s$ is even,} \\
4\cdot17^{(s-2)/2}\cdot5^{(k-s+1)/2}+8\cdot17^{(s-2)/2}\cdot5^{(k-s-1)/2},  & \text{if $s$ is even and $k-s$ is odd,} \\
17^{(s-1)/2}\cdot5^{(k-s+1)/2}+2\cdot 17^{(s-1)/2}\cdot5^{(k-s-1)/2}, & \text{if $s$ is odd and $k-s$ is odd,} \\
2\cdot17^{(s-1)/2}\cdot5^{(k-s)/2}+17^{(s-1)/2}\cdot5^{(k-s)/2},  & \text{if $s$ is odd and $k-s$ is even}
\end{cases} \\
&\leq g(k,s),
\end{aligned}
\end{equation*}
a contradiction. This completes the proof of Claim~\ref{cl:GRk4+egeneral-1}.
\end{proof}

We define $R$ to be a 2-edge-coloring of $K_t$ with $V(R)=\{v_1, v_2, \ldots, v_t\}$ and $c(v_iv_j)=c(V_i, V_j)$ for any $1\leq i<j\leq t$. Note that if $R$ contains a 2-edge-colored subgraph $H$, then $G$ also contains a copy of $H$ (in fact, $G$ contains a blow-up of $H$). For each $i\in [t]$, let $N^r_i:=\{j\in [t]\setminus \{i\}\colon\, c(v_iv_j) \mbox{ is red}\}$, $N^b_i:=\{j\in [t]\setminus \{i\}\colon\, c(v_iv_j) \mbox{ is blue}\}$, $d^r_i:=\left|N^r_i\right|$ and $d^b_i:=\left|N^b_i\right|$. By Claim \ref{cl:GRk4+egeneral-1} and the minimality of $t$, we have $d^r_i\geq 1$ and $d^b_i\geq 1$ for every $i\in [t]$. We claim that at least one of red and blue is among the first $s$ colors. Indeed, if both red and blue are among the last $k-s$ colors, then $R$ contains no monochromatic copy of $K_3$. So $t\leq R(K_3, K_3)-1=5$. Moreover, for every $i\in [t]$, since $d^r_i\geq 1$ and $d^b_i\geq 1$, there is no red edge and no blue edge within $V_i$ in $G$. By the induction hypothesis, we have $n=\sum^t_{i=1}|V_i|\leq 5\cdot g(k-2,s)\leq g(k,s)$, a contradiction.

Let $\mathcal{R}:=\{i\in [t]\colon\, G[V_i] \mbox{ contains a red edge}\}$ and $\mathcal{B}:=\{i\in [t]\colon\, G[V_i] \mbox{ contains a blue}$ edge$\}$. Let $x_0:=\left|[t]\setminus (\mathcal{R}\cup \mathcal{B})\right|$, $x_1:=\left|\mathcal{R}\bigtriangleup \mathcal{B}\right|$ and $x_2:=\left|\mathcal{R}\cap \mathcal{B}\right|$, so $t=x_0+x_1+x_2$. We have the following simple facts.

\noindent\begin{fact}\label{fa:GRk4+egeneral-2} $ $
\begin{itemize}
\item[{\rm (1)}] For any $i\in \mathcal{R}$ (resp., $i\in \mathcal{B}$), we have that $v_i$ is not contained in any red copy of $K_3$ (resp., blue copy of $K_3$) in $R$.

\item[{\rm (2)}] For any $i, j\in \mathcal{R}$ (resp., $i, j\in \mathcal{B}$) with $i\neq j$, we have that $c(V_i, V_j)$ is blue (resp., red).

\item[{\rm (3)}] For any $i\in \mathcal{R}$ (resp., $i\in \mathcal{B}$), we have $d^r_i\leq 3$ (resp., $d^b_i\leq 3$).

\item[{\rm (4)}] For any $i\in [t]$, we have $d^r_i\leq 8$ and $d^b_i\leq 8$.

\item[{\rm (5)}] For any $i\in [t]$, $G[V_i]$ contains neither a red copy of $K_3$ nor a blue copy of $K_3$.

\item[{\rm (6)}] $x_2\leq 1$.
\end{itemize}
\end{fact}

\begin{proof} By the symmetry of red and blue, we will only prove the red case for (1)--(5). Note that if red is one of the last $k-s$ colors, then Fact \ref{fa:GRk4+egeneral-2} holds clearly. So we may assume that red is one of the first $s$ colors.

(1) If there exists an $i\in \mathcal{R}$ such that $v_i$ is contained in a red copy of $K_3$ in $R$, say $v_iv_jv_{\ell}$, then in order to avoid a red copy of $K_4+e$, we have that $c(V_i\cup V_j\cup V_{\ell}, V(G)\setminus (V_i\cup V_j\cup V_{\ell}))$ is blue. By the minimality of $t$, we have $t=2$, contradicting Claim \ref{cl:GRk4+egeneral-1}.

(2) If there exist some $i, j\in \mathcal{R}$ with $i\neq j$ such that $c(V_i, V_j)$ is red, then for avoiding a red copy of $K_4+e$, we have that $c(V_i\cup V_j, V(G)\setminus (V_i\cup V_j))$ is blue. By the minimality of $t$, we have $t=2$, contradicting  Claim \ref{cl:GRk4+egeneral-1}.

(3) If there exists an $i\in \mathcal{R}$ such that $d^r_i\geq 4$, then $\{v_j\colon\, j\in N^r_i\}$ forms a blue copy of $K_{d^r_i}$ by (1). In order to avoid a blue copy of $K_4+e$, we have $d^r_i= 4$ and $c(\bigcup_{j\in N^r_i}V_j, \bigcup_{\ell \in [t]\setminus N^r_i}V_{\ell})$ is red. By the minimality of $t$, we have $t=2$, contradicting Claim \ref{cl:GRk4+egeneral-1}.

(4) Suppose $d^r_i\geq 9$ for some $i\in [t]$. In order to avoid a red copy of $K_4+e$, there is no red copy of $K_3$ in $R[\{v_j\colon\, j\in N^r_i\}]$. Since $R(K_3, K_4+e)=9$, there is a blue copy of $K_4+e$ (and thus a blue copy of $K_3$), a contradiction.

(5) Suppose that $G[V_i]$ contains a red copy of $K_3$ for some $i\in [t]$. Since $d^{r}_i\geq 1$, we may assume that $c(V_i, V_j)$ is red for some $j\in [t]\setminus \{i\}$. In order to avoid a red copy of $K_4+e$, we have that $c(V_i\cup V_j, V(G)\setminus (V_i\cup V_j))$ is blue. By the minimality of $t$, we have $t=2$, contradicting Claim \ref{cl:GRk4+egeneral-1}.

(6) If $x_2=\left|\mathcal{R}\cap \mathcal{B}\right|\geq 2$, then we can derive a contradiction by (2).
\end{proof}

We divide the rest of the proof into two cases according to where red and blue are in the list of colors.

\medskip\noindent
{\bf Case 1.} Red is among the first $s$ colors and blue is among the last $k-s$ colors.
\vspace{0.05cm}

In this case, there is no red copy of $K_4+e$ and no blue copy of $K_3$ in $G$. Since $R(K_4+e, K_3)=9$, we have $4\leq t\leq 8$. Recall that $d^r_i\geq 1$ and $d^b_i\geq 1$ for every $i\in [t]$. So there is no blue edge within each $V_i$. Thus $|\mathcal{B}|=0$, $x_1=|\mathcal{R}|$, $x_2=0$ and $x_0=t-x_1$. We claim that $x_1\leq 2$, since otherwise if $|\mathcal{R}|\geq 3$, then there is a blue copy of $K_3$ by Fact \ref{fa:GRk4+egeneral-2} (2).

For each $i\in \mathcal{R}$, $G[V_i]$ contains no red copy of $K_3$ by Fact \ref{fa:GRk4+egeneral-2} (5). By the induction hypothesis, we have
\begin{equation*}
\begin{aligned}
|V_i|&\leq g(k-1, s-1) \\
&=
\begin{cases}
17^{(s-1)/2}\cdot5^{(k-s)/2},  & \text{if $s-1$ is even ($s$ is odd) and $k-s$ is even,} \\
2\cdot17^{(s-1)/2}\cdot5^{(k-s-1)/2},  & \text{if $s-1$ is even ($s$ is odd) and $k-s$ is odd,} \\
8\cdot17^{(s-2)/2}\cdot5^{(k-s-1)/2},  & \text{if $s-1$ is odd ($s$ is even) and $k-s$ is odd,} \\
4\cdot 17^{(s-2)/2}\cdot5^{(k-s)/2}, & \text{if $s-1$ is odd ($s$ is even) and $k-s$ is even}
\end{cases} \\
&\leq \frac{1}{4}g(k,s).
\end{aligned}
\end{equation*}

For each $i\in [t]\setminus (\mathcal{R}\cup \mathcal{B})$, by the induction hypothesis, we have

\begin{equation*}
\begin{aligned}
|V_i| &\leq g(k-2, s-1)
&=
\begin{cases}
17^{(s-1)/2}\cdot5^{(k-s-1)/2},  & \text{if $s-1$ is even and $k-s-1$ is even}\\
                               & \text{($s$ is odd and $k-s$ is odd),} \\
2\cdot17^{(s-1)/2}\cdot5^{(k-s-2)/2},  & \text{if $s-1$ is even and $k-s-1$ is odd}\\
                               & \text{($s$ is odd and $k-s$ is even),} \\
8\cdot17^{(s-2)/2}\cdot5^{(k-s-2)/2},  & \text{if $s-1$ is odd and $k-s-1$ is odd} \\
                               & \text{($s$ is even and $k-s$ is even),} \\
4\cdot 17^{(s-2)/2}\cdot5^{(k-s-1)/2}, & \text{if $s-1$ is odd and $k-s-1$ is even} \\
                               & \text{($s$ is even and $k-s$ is odd)}
\end{cases} \\
&\leq \frac{1}{8}g(k,s).
\end{aligned}
\end{equation*}

Thus $n\leq (x_1/4+x_0/8)g(k,s)$. It suffices to prove that $x_1/4+x_0/8\leq 1$. If $x_1\leq 8-t$, then $x_1/4+x_0/8=(2x_1+x_0)/8=(x_1+t)/8\leq 1$. Thus we may assume $x_1\geq 8-t+1$. Recall that we have $t\leq 8$ and $x_1\leq 2$ in this case. So $|\mathcal{R}|=x_1\geq 1$ and $7\leq t\leq 8$. For any $i\in \mathcal{R}$, we have $d^r_i\leq 2$ for avoiding a blue copy of $K_3$ and by Fact \ref{fa:GRk4+egeneral-2} (1). Thus $d^b_i\geq 4$. Since there is no blue copy of $K_3$, we have that $\{v_j\colon\, j\in N^b_i\}$ forms a red copy of $K_{d^b_i}$. Then $c(\bigcup_{j\in N^b_i}V_j, \bigcup_{\ell \in [t]\setminus N^b_i}V_{\ell})$ is blue. By the minimality of $t$, we have $t=2$, contradicting Claim \ref{cl:GRk4+egeneral-1}.

\medskip\noindent
{\bf Case 2.} Both red and blue are among the first $s$ colors.
\vspace{0.05cm}

In this case, we have $4\leq t\leq 17$ since $R(K_4+e, K_4+e)=18$. Moreover, we have $s\geq 2$ and thus $g(k,s)\geq 34$ (recall that $k\geq 3$). By the induction hypothesis, for every $i\in [t]\setminus (\mathcal{R}\cup \mathcal{B})$, we have $|V_i|\leq g(k-2,s-2)= \frac{1}{17}g(k,s)$.
For any $i\in [t]$, $G[V_i]$ contains neither a red copy of $K_3$ nor a blue copy of $K_3$ by Fact \ref{fa:GRk4+egeneral-2} (5). Thus for each $i\in \mathcal{R}\cap \mathcal{B}$, by the induction hypothesis, we have $|V_i|\leq g(k, s-2)= \frac{5}{17}g(k,s)$. And for each $i\in \mathcal{R}\bigtriangleup \mathcal{B}$, we have
\begin{equation*}
\begin{aligned}
|V_i| &\leq g(k-1, s-2)&=
\begin{cases}
17^{(s-2)/2}\cdot5^{(k-s+1)/2},  & \text{if $s-2$ is even and $k-s+1$ is even,} \\
2\cdot17^{(s-2)/2}\cdot5^{(k-s)/2},  & \text{if $s-2$ is even and $k-s+1$ is odd,} \\
8\cdot17^{(s-3)/2}\cdot5^{(k-s)/2},  & \text{if $s-2$ is odd and $k-s+1$ is odd,} \\
4\cdot 17^{(s-3)/2}\cdot5^{(k-s+1)/2}, & \text{if $s-2$ is odd and $k-s+1$ is even}
\end{cases} \\
&\leq \frac{5}{34}g(k,s).
\end{aligned}
\end{equation*}

Thus $n\leq (5x_2/17+5x_1/34+x_0/17)g(k,s)$. It suffices to prove that $10x_2+5x_1+2x_0=2t+8x_2+3x_1\leq 34$.

\noindent\begin{claim}\label{cl:GRk4+egeneral-3} $x_2=0$.
\end{claim}

\begin{proof} By Fact \ref{fa:GRk4+egeneral-2} (6), we have $x_2\leq 1$. For a contradiction, suppose $\mathcal{R}\cap \mathcal{B}=\{1\}$. By Fact \ref{fa:GRk4+egeneral-2} (3), we have $d^r_1\leq 3$ and $d^b_1\leq 3$, so $t\leq 7$. If $t\leq 5$, then $2t+8x_2+3x_1\leq 10+8+12\leq 34$. If $6\leq t\leq 7$, then we may assume that $d^r_1= 3$ without loss of generality, say $N^r_1=\{2, 3, 4\}$. By Fact \ref{fa:GRk4+egeneral-2} (1), we have that $c(v_2v_3)=c(v_3v_4)=c(v_2v_4)$ is blue. By Fact \ref{fa:GRk4+egeneral-2} (1) and (2), we have $2, 3, 4\notin \mathcal{R}\cup \mathcal{B}$. Thus $x_1\leq t-4$, so $2t+8x_2+3x_1\leq 8+5t-12\leq 34$.
\end{proof}

\noindent\begin{claim}\label{cl:GRk4+egeneral-4} $|\mathcal{R}|\leq 3$ and $|\mathcal{B}|\leq 3$. If $|\mathcal{R}|= 3$ (resp., $|\mathcal{B}|= 3$), then $|\mathcal{B}|\leq 1$ (resp., $|\mathcal{R}|\leq 1$).
\end{claim}

\begin{proof} If $|\mathcal{R}|\geq 4$ (resp., $|\mathcal{B}|\geq 4$), then $G$ contains a blue (resp., red) $K_{2,2,2,2}$ by Fact \ref{fa:GRk4+egeneral-2} (2). This implies a monochromatic copy of $K_4+e$ in $G$. Thus $|\mathcal{R}|\leq 3$ and $|\mathcal{B}|\leq 3$.

If $|\mathcal{R}|= 3$ and $2\leq |\mathcal{B}|\leq 3$, then $R[\{v_i\colon\, i\in\mathcal{R}\}]$ and $R[\{v_i\colon\, i\in\mathcal{B}\}]$ form a blue clique and a red clique (by Fact \ref{fa:GRk4+egeneral-2} (2)), respectively. By Fact \ref{fa:GRk4+egeneral-2} (1), for any $i\in \mathcal{R}$ (resp., $i\in \mathcal{B}$), there is at most one red (resp., blue) edge between $v_i$ and $\{v_j\colon\, j\in\mathcal{B}\}$ (resp., $\{v_j\colon\, j\in\mathcal{R}\}$). Thus there are at most $|\mathcal{R}|+|\mathcal{B}|< |\mathcal{R}||\mathcal{B}|$ edges between $\{v_i\colon\, i\in\mathcal{R}\}$ and $\{v_i\colon\, i\in\mathcal{B}\}$, a contradiction. Therefore, if $|\mathcal{R}|= 3$, then $|\mathcal{B}|\leq 1$, and similarly, if $|\mathcal{B}|= 3$, then $|\mathcal{R}|\leq 1$.
\end{proof}

By Claims \ref{cl:GRk4+egeneral-3} and \ref{cl:GRk4+egeneral-4}, we have $x_2=0$ and $x_1=|\mathcal{R}|+|\mathcal{B}|\leq 4$. If $t\leq 11$, then $2t+8x_2+3x_1\leq 22+0+12=34$. If $13\leq t\leq 17$, then $|\mathcal{R}|=|\mathcal{B}|=0$ by Fact \ref{fa:GRk4+egeneral-2} (3) and (4), so $2t+8x_2+3x_1\leq 34+0+0=34$. Thus $t=12$. We have $x_1=|\mathcal{R}|+|\mathcal{B}|= 4$; otherwise $2t+8x_2+3x_1\leq 24+0+9\leq34$. Then we further have $|\mathcal{R}|\geq 1$ and $|\mathcal{B}|\geq 1$ by Claim \ref{cl:GRk4+egeneral-4}. Without loss of generality, let $1\in \mathcal{R}$, $2\in \mathcal{B}$ and let $c(V_1, V_2)$ be blue. Moreover, by Fact \ref{fa:GRk4+egeneral-2} (3) and (4), we have $d^r_1=3$, $d^b_1=8$, $d^b_2=3$ and $d^r_2=8$. We may further assume that $c(V_1, V_3\cup V_4\cup \cdots \cup V_9)$ is blue. By Fact \ref{fa:GRk4+egeneral-2} (1), we have $c(V_2, V_3\cup V_4\cup \cdots \cup V_9)$ is red. Since $R(K_3, K_3)=6$, there is either a red copy of $K_3$ or a blue copy of $K_3$ in $R[\{v_3, v_4, \ldots, v_9\}]$. Then there is either a red copy of $K_4+e$ or a blue copy of $K_4+e$ in $G$, a contradiction.
\end{proof}

\section{Concluding remarks}
\label{sec:ch-remark}

In Section \ref{sec:ch-proof-RMk3k3}, we studied the maximum number (denoted by $f_k(n,H)$) of edges that are not contained in any rainbow triangle or monochromatic copy of $H$. There we showed that $f_k(n,H)\geq t(n, GR_{k-1}(\mathscr{H})-1)$, where $\mathscr{H}$ is the set of homomorphic copies of $H$. Let $f'_k(n,H)$ be the maximum number of edges not contained in any monochromatic copy of $H$ over all Gallai-$k$-colorings of $K_n$. Then we clearly have $f'_k(n, H)\leq f_k(n, H)$. Using the sharpness example constructed in the proof of Lemma \ref{le:gr*bound} (2), we can also show that $f'_k(n,H)\geq t(n, GR_{k-1}(\mathscr{H})-1)$. Thus we have $t(n, GR_{k-1}(\mathscr{H})-1)\leq f'_k(n,H)\leq f_k(n,H)$. An interesting and natural question is for which graphs $H$ the equality  $f'_k(n,H)= f_k(n,H)$ holds.

Another problem related to Section \ref{sec:ch-proof-RMk3k3} is to determine the maximum number nim$_k(n, H)$ of edges not contained in any monochromatic copy of $H$ over all $k$-edge-colorings of $K_n$. As remarked in \cite{LiPS}, if the Erd\H{o}s-S\'{o}s conjecture holds for a tree $T$ (i.e., $ex(n, T)\leq (|V(T)|-2)n/2$), then for each $n\geq k^2(|V(T)|-1)^2$ with $(|V(T)|-1)\mid n$, we have nim$_k(n, T)\geq (k-1)ex(n, T)$. In fact, when $T$ is a star, we can prove the above statement for all $n\geq k^2(|V(T)|-1)^2$. Let $H$ be an $n$-vertex $K_{1,h}$-free graph with $ex(n, K_{1,h})$ edges. Note that the maximum degree of $H$ is at most $h-1$. For every $i\in [k-1]$, let $f_i\colon\, V(H)\rightarrow [n]$ be an arbitrary bijection and let $H_i$ be the graph obtained by mapping $H$ on $[n]$ via $f_i$. Let $H^{\ast}$ be the graph with vertex set $[n]$ and edge set $\bigcup_{i\in [k-1]}E(H_i)$. Note that $\Delta(H^{\ast})\leq (k-1)(h-1)$. For any vertex $u$, there is a vertex $v$ that is at distance at least three from $u$ in $H^{\ast}$ since $n>\Delta(H^{\ast})^2+1$. If there is an edge $e$ incident with $u$ or $v$ such that $e\in E(H_i)\cap E(H_j)$ for some $1\leq i\neq j\leq k-1$, then after switching $u$ and $v$ in $f_i$, we claim that there is no edge $e'$ incident with $u$ or $v$ satisfying $e'\in E(H_i)\cap E(H_{\ell})$ for any $\ell\in [k-1]\setminus \{i\}$. Otherwise, suppose that there is an edge $vw\in E(H_i)\cap E(H_{\ell})$ after switching $u$ and $v$ in $f_i$. This implies that before switching $u$ and $v$ in $f_i$, we have $vw\in E(H_{\ell})$ and $uw\in E(H_i)$. Thus $uwv$ is a path of length two in $H^{\ast}$, contradicting the fact that $v$ is at distance at least three from $u$. Thus we can repeat this process to obtain a graph with no edge $e$ such that $e\in E(H_i)\cap E(H_j)$ for some $1\leq i\neq j\leq k-1$. Hence, we can color $K_{n}$ with $c(e)=i$ if $e\in E(H_i)$ for each $i\in [k-1]$ and $c(e)=k$ otherwise. Thus nim$_k(n, K_{1,h})\geq \sum_{i\in [k-1]}|E(H_i)|=(k-1)ex(n, K_{1,h})$.

Moreover, let $G$ be a $k$-edge-coloring of $K_{n}$ with nim$_k(n, K_{1,h})$ edges not contained in any monochromatic copy of $K_{1,h}$. For $i\in [k]$, let $G_i$ (resp., $G^{nim}_i$) denote the spanning subgraph of $G$ with edge set $E(G_i)=\{e\in E(G)\colon\, c(e)=i\}$ (resp., $E(G^{nim}_i)=\{e\in E(G)\colon\, e$ is not contained in any monochromatic copy of $K_{1,h}, c(e)=i\}$) and let $V_i=\{v\in V(G)\colon\, d_{G_i}(v)\geq h\}$. If $n>k(h-1)+1$, then $\bigcup_{i\in [k]}V_i = V(G)$, and every vertex of $V_i$ is an isolated vertex in $G^{nim}_i$ for every $i\in [k]$. Since $ex(n, K_{1,h})=\left\lfloor(h-1)n/2\right\rfloor$, we have nim$_k(n, K_{1,h})=\sum_{i\in [k]}e(G^{nim}_i)\leq\sum_{i\in [k]} ex\left(n-|V_i|, K_{1,h}\right)\leq ex\left(\sum_{i\in [k]}(n-|V_i|), K_{1,h}\right)\leq ex((k-1)n, K_{1,h})$. Note that $ex((k-1)n, K_{1,h})=(k-1)ex(n, K_{1,h})+\eta$, where $\eta=\left\lfloor (k-1)/2\right\rfloor$ if $h$ is even and $n$ is odd, and $\eta=0$ otherwise. Therefore, for $n\geq k^2h^2$, if $h$ is even and $n$ is odd, then $(k-1)ex(n, K_{1,h})\leq$ nim$_k(n$ $K_{1,h}) \leq (k-1)ex(n, K_{1,h})+\left\lfloor (k-1)/2\right\rfloor$, and otherwise, we have nim$_k(n, K_{1,h})= (k-1)ex(n, K_{1,h})$. In particular, we have the following result in the case $k=2$, which partly answers a problem of Keevash and Sudakov \cite{KeSu} in the special case when $H$ is a star.

\noindent\begin{proposition}\label{prop:nimstar} For $n$ sufficiently large, we have {\rm nim}$_2(n, K_{1,h})= ex(n, K_{1,h})$.
\end{proposition}

In Section \ref{sec:ch-proof-multiplicity}, we studied the minimum number of copies of $H$ over all Gallai-$k$-colorings of $K_{n}$. Given an arbitrary $k$-edge-coloring $G$ of $K_n$, let $r_k(K_3, n)$ and $m_k(H,n)$ be the number of rainbow triangles and monochromatic copies of $H$ in $G$, respectively. It is interesting to consider the behavior of $r_k(K_3,n)+m_k(H,n)$. Clearly if $k\leq 2$, then $r_k(K_3,n)+m_k(H,n)=m_k(H,n)$, and if $G$ is rainbow, then $r_k(K_3,n)+m_k(H,n)=\binom{n}{3}$. However, the general behavior of $r_k(K_3,n)+m_k(H,n)$ seems difficult to determine.

Finally, we pose two conjectures. Note that we have shown that Conjecture \ref{conj:multiplicity1} below holds for the following cases: (1) $k=3$ and $n$ sufficiently large, (2) $k\geq 3$ and $n=GR_k(K_3)$, (3) $k$ is odd and $GR_k(K_3)\leq n\leq GR_k(K_3)+5^{(k-1)/2}-1$.

\noindent\begin{conjecture}\label{conj:multiplicity1} For $n\geq GR_k(K_3)$, we write $n=5^{\lfloor(k-1)/2\rfloor}m+r$, where $m$ and $r$ are nonnegative integers with $0\leq r\leq 5^{\lfloor(k-1)/2\rfloor}-1$. Then
$$g_{k}(K_3,n)= \left\{
                \begin{aligned}
                & r\binom{m+1}{3}+\left(5^{(k-1)/2}-r\right)\binom{m}{3}, & & \mbox{if $k$ is odd},\\
                & rM_2(K_3, m+1)+\left(5^{(k-2)/2}-r\right)M_2(K_3, m), & & \mbox{if $k$ is even}.
              \end{aligned}
           \right.$$
\end{conjecture}

\noindent\begin{conjecture}\label{conj:edge} For integers $k\geq 2$, we have $f_{k}(n, K_3)=t(n, GR_{k-1}(K_3)- 1)$.
\end{conjecture}

\noindent{\bf Note.} We recently discovered that Theorem \ref{th:GRk4+e} has been proved by Su and Liu \cite{SuLi} and Zhao and Wei \cite{ZhWe} independently.

\section*{Acknowledgement}

The authors are grateful to the anonymous referee for valuable comments, suggestions and corrections which improved the presentation of this paper.

\end{document}